\documentclass[12pt]{amsart}
\usepackage{graphicx}
\usepackage{amsmath,amsthm, amsfonts,amssymb, stmaryrd,yfonts,pxfonts,pifont,eufrak, bbm}
\newtheorem{Cor}{Corollary}
 \newtheorem{Lemma}{Lemma}
 
 \newtheorem{ex}{Example}
 \newtheorem{Proposition}{Proposition}
 \theoremstyle{definition}
 
 \theoremstyle{remark}
 \newtheorem{Remark}[Lemma]{Remark}
 \numberwithin{equation}{subsection}

\begin{document}
\title[DYNAMICS OF FINITELY GENERATED NON-AUTONOMOUS SYSTEMS]{DYNAMICS OF FINITELY GENERATED NON-AUTONOMOUS SYSTEMS}%
\author{MANISH RAGHAV AND PUNEET SHARMA}
\address{Department of Mathematics, I.I.T. Jodhpur, NH 65, Karwar, Jodhpur-342037, INDIA}%
\email{puneet@iitj.ac.in, manishrghv@gmail.com }%

\thanks{The first author thanks National Board for Higher Mathematics (NBHM) for financial support.}%

\subjclass[2010]{37B20, 37B55, 54H20}

\keywords{non-autonomous dynamical systems, transitivity, weakly
mixing, topological mixing, topological entropy, Li-Yorke chaos}

\begin{abstract}
In this paper, we discuss dynamical behavior of a non-autonomous system generated by a finite family $\mathbb{F}$. In the process, we relate the dynamical behavior of the non-autonomous system generated by the family $\mathbb{F}=\{f_1,f_2,\ldots,f_k\}$ with the dynamical behavior of the system $(X,f_k\circ f_{k-1}\circ\ldots\circ f_1)$. We discuss properties like minimality, equicontinuity, proximality and various forms of sensitivities for the two systems. We derive conditions under which the dynamical behavior of $(X,f_k\circ f_{k-1}\circ\ldots\circ f_1)$ is carried forward to $(X,\mathbb{F})$ (and vice-versa). We also give examples to illustrate the necessity of the conditions imposed.
\end{abstract}
\maketitle

\section{INTRODUCTION}
For a long time, dynamical systems have been used to determine the long term behavior of various natural and physical systems around us. The theory provides an alternative mathematical approach to investigate the qualitative behavior of any natural or physical system under investigation. The theory helps not only to determine long term behavior of the underlying system but also predicts any long term complexities that may arise in the given system. As a result the theory has found applications in a variety of fields such as population dynamics, control engineering and computational neuroscience\cite{iz,nar,zh}. Although dynamical systems have been used to approximate many natural and physical systems with sufficient accuracy, theory of autonomous systems is used to predict the long term behavior of the system. However, as a time varying governing rule provides a better insight of the underlying system, approximating the given system using non-autonomous dynamical systems provides a better approximation of the underlying system. Consequently, it is important to develop the theory of non-autonomous dynamical systems. The problem has been addressed in recent times and some interesting results have been obtained. In \cite{sk1}, the authors investigate the topological entropy of the non-autonomous system generated by a equicontinuous family or a uniformly convergent sequence of self maps. In \cite{sk2}, the authors investigate the minimality conditions for a non-autonomous system on in a compact Hausdorff setting. In \cite{pm1}, the authors investigate the dynamics of a non-autonomous system generated by a finite collection of maps. In particular, they investigate properties like existence of periodic points, various forms of mixing, topological entropy and Li-Yorke chaos for the non-autonomous system. Some more studies on the topic have been conducted and can be found in the literature\cite{bp,jd,pm2}. We now give some of the basic definitions and concepts required. \\

Let $(X,d)$ be a compact metric space and let $\mathbb{F}=\{f_1,f_2,\ldots,f_k\}$ be a finite collection of continuous self maps on $X$. For any given $x_0\in X$, the family $\mathbb{F}$ gives rise to a non-autonomous dynamical system via the rule $x_{n}= f_n(x_{n-1})$. Throughout this paper, the non-autonomous system generated by the family $\mathbb{F}$ will be denoted by $(X,\mathbb{F})$. The set $\{ f_n \circ f_{n-1} \circ \ldots \circ f_1(x_0):n\in\mathbb{N}\}$ defines the orbit of the point $x_0$. For notational convenience, let $\omega_n(x) = f_n\circ f_{n-1}\circ \ldots \circ f_1(x)$ (the state of the system after $n$ iterations). \\

The system $(X,\mathbb{F})$ is said to be \textit{minimal} if every point has a dense orbit. The system $(X,\mathbb{F})$ is equicontinuous if for each $\epsilon>0$, there exists $\delta>0$ such that $d(x,y)<\delta$ implies $d(\omega_n(x),\omega_n(y))<\epsilon$ for all $n\in\mathbb{N},~~x,y\in X$. The system $(X,\mathbb{F})$ is \textit{sensitive} if there exists a $\delta>0$ such that for each $x\in X$ and each neighborhood $U$ of $x$, there exists $n\in \mathbb{N}$ such that $diam(\omega_n(U))>\delta$. If there exists $K>0$ such that $diam(\omega_n(U))>\delta$ $~~\forall n\geq K$, then the system is \textit{cofinitely sensitive}. Any pair $(x,y)$ is called proximal for $(X,\mathbb{F})$ if $\liminf\limits_{n\rightarrow \infty} d(\omega_n(x),\omega_n(y))=0$. It may be noted that the points in the diagonal set $\Delta=\{(x,x):x\in X\}$ are trivially proximal. The system $(X,\mathbb{F})$ is called distal if it does not have any non-trivial proximal pairs. A set $S$ is called \textit{scrambled} if for any pair of distinct points $x,y\in S$, $\limsup\limits_{n\rightarrow \infty} d(\omega_n(x),\omega_n(y))>0$ but $\liminf\limits_{n\rightarrow \infty} d(\omega_n(x),\omega_n(y))=0$.  In addition, if $\limsup\limits_{n\rightarrow \infty} d(\omega_n(x),\omega_n(y))>\delta$ and $\liminf\limits_{n\rightarrow \infty} d(\omega_n(x),\omega_n(y))=0$, for any pair of distinct points $x,y\in S$, then $S$ is called a $\delta$-scrambled set. The system $(X,\mathbb{F})$ is said to be Li-Yorke sensitive if there exists $\delta>0$ such that for each $x\in X$ and each neighborhood $U$ of $x$ there exists $y\in U$ such that $(x,y)$ is a $\delta$-scrambled set. The system $(X,\mathbb{F})$ is said to be \textit{Li-Yorke chaotic} if it contains an uncountable scrambled set. In case the $f_n$'s coincide, the above definitions coincide with the known notions of an autonomous dynamical system. See \cite{bc,bs,de} for details.\\

In this paper, we investigate the dynamical properties of the non-autonomous system generated by a finite collection of self maps. In the process, we relate the dynamical behavior of the non-autonomous system generated with the dynamics of the autonomous system $(X,f_k\circ f_{k-1}\circ\ldots\circ f_1)$. We relate properties like minimality, equicontinuity, proximality, various forms of sensitivities and Li-Yorke chaos for the two systems. We derive conditions under which the dynamical behavior of the system $(X,f_k\circ f_{k-1}\circ\ldots\circ f_1)$ is carried forward to the non-autonomous system (and vice-versa). We also give examples to investigate the necessity of the conditions imposed. \\

\section{Main Results}

Let $(X,\mathbb{F})$ be the non-autonomous system generated by the finite family $\mathbb{F}=\{f_1,f_2,\ldots,f_k\}$ and let $f=f_k\circ f_{k-1}\circ\ldots\circ f_1$. In this section, we relate the dynamical properties of the two systems $(X,\mathbb{F})$ and $(X,f)$. Throughout this section, the maps $f_i$ are assumed to be surjective.

\begin{Proposition}
$(X,\mathbb{F})$ is equicontinuous $\Leftrightarrow (X,f)$ is equicontinuous.
\end{Proposition}
\begin{proof}
Let $(X,\mathbb{F})$ be equicontinuous and let $\epsilon>0$ be given. As $(X,\mathbb{F})$ is equicontinuous, there exists $\delta>0$ such that $d(x,y)<\delta$ implies $d(\omega_n(x),\omega_n(y))<\epsilon~~\forall n\in\mathbb{N}$. In particular, $d(\omega_{nk}(x),\omega_{nk}(y))<\epsilon~~\forall n\in\mathbb{N}$ and hence $(X,f)$ is equicontinuous.

Conversely, let $(X,f)$ be equicontinuous and let $\epsilon>0$ be given. Then, as the family $\{g_r=f_r\circ \cdots f_2\circ f_1 : r=1,2,\ldots,k\}$ is a finite family of continuous maps, there exists $\eta>0$ ($\eta<\epsilon$) such that $d(x,y)<\eta$ implies $d(g_r(x),g_r(y))<\epsilon$ (for $r=1,2,\ldots,k$). As $(X,f)$ is equicontinuous, there exists $\delta>0$ ($\delta<\eta$) such that $d(x,y)<\delta$ implies $d(\omega_{nk}(x),\omega_{nk}(y))<\eta~~\forall n\in\mathbb{N}$. Consequently, $d(x,y)<\delta$ ensures $d(g_r(\omega_{nk}(x)),g_r(\omega_{nk}(y)))<\epsilon$ for all $r\in \{1,2,\ldots,k\}$ and $n\in\mathbb{Z}^+$. As any point $\omega_m(x)$ can be written as $g_r(\omega_{nk}(x))$ for some $r\in \{1,2,\ldots,k\}$ and $n\in\mathbb{Z}^+$, we have $d(\omega_m(x),\omega_m(y))<\epsilon$ for all $m\in\mathbb{N}$ and hence $(X,\mathbb{F})$ is equicontinuous.
\end{proof}

\begin{Proposition}\label{mini}
If $X$ is connected then, $(X,\mathbb{F})$ is minimal $\Leftrightarrow$ $(X,f)$ is minimal.
\end{Proposition}

\begin{proof}
Let $(X,\mathbb{F})$ be minimal and let $x\in X$. As orbit of $x$ is dense in $(X,\mathbb{F})$, for each $y\in X$ there exists a sequence $(m_i)$ of natural numbers and $r\in\{1,2,\ldots,k\}$ such that $\omega_{km_i+r}(x)$ converges to $y$ (follows from the fact that the generating family $\mathbb{F}$ is finite). For any $a,b\in X$, we say $a$ is related to $b$ in $\mathbb{F}$-sense (denoted as $a\mathbb{F}b$) if there exists $m \in \{1,2,\ldots,k\}$ and sequences $(s_i)$ and $(t_i)$ of natural numbers such that $(\omega_{ks_i+m}(x),\omega_{kt_i+m}(x))$ converges to $(a,b)$ (in the product topology). Note that the relation defines an equivalence relation on $X$ and hence partitions $X$ into $k$ disjoint sets $C_1,C_2,\ldots,C_k$ of $X$. Further, as each $C_k$ is closed (and hence clopen), connectedness of $X$ implies $C_1=C_2=\ldots=C_k=X$ and hence orbit of $x$ is dense in $(X,f)$.

Conversely, as orbit of any point $x$ under $(X,f)$ is contained in orbit of $x$ in $(X,\mathbb{F})$, minimality of $(X,f)$ implies minimality of $(X,\mathbb{F})$ and hence the proof of converse is complete.
\end{proof}

\begin{Remark}
The above result establishes the equivalence of minimality for the two systems $(X,\mathbb{F})$ and $(X,f)$ when the space $X$ is connected. Although the proof of converse is trivial, the proof for the forward part partitions the space $X$ into $k$ (atmost) disjoint non-empty clopen sets $C_r$, where $C_r$ is the set of limit points of the sequence $(\omega_{nk+r}(x))$~~ ($r=1,2,\ldots,k$). Consequently, if $X$ is connected, the generated sets coincide which in turn implies the denseness of orbit of $x$ (for $(X,f)$) and hence minimality of the two systems is equivalent. However, the equivalence holds good only when the space $X$ is connected and fails to hold good in the absence of the stated condition. We now give an example in support of our claim.
\end{Remark}

\begin{ex}\label{min}
Let $S_r=\{re^{i\theta}: 0\leq\theta\leq 2\pi\}$ ($r=1,2$) and let $X=S_1\cup S_2$. Let $\alpha\in\mathbb{R}$ be an irrational multiple of $2\pi$. Define $f_1,f_2:X\rightarrow X$ as

$f_1(x) = \left\{%
\begin{array}{ll}
            (r+1)e^{i(\theta+\alpha)}  & \text{for~~} r=1 \\
           (r-1)e^{i(\theta+2\alpha)} & \text{for~~} r=2 \\
\end{array} \right.$

$f_2(x) = \left\{%
\begin{array}{ll}
            (r+1)e^{i(\theta+2\alpha)}  & \text{for~~} r=1 \\
           (r-1)e^{i(\theta+\alpha)} & \text{for~~} r=2 \\
\end{array} \right.$

It may be noted that both $f_1$ and $f_2$ map $S_1$ to $S_2$ (and vice-versa) with an additional rotation of angle $\alpha$ (or $2\alpha$). Further as orbit of any point $x$ in $S_i$ is a rotation on $S_i$ by angle $2\alpha$ (or $4\alpha$) at even iterates and visits the other component of the space $X$ via a rotation by angle $2\alpha$ (or $4\alpha$) at odd iterates, the system $(X,\mathbb{F})$ is minimal. However, as $f_2\circ f_1$ keeps each $S_i$ invariant, the system $(X,f)$ is not minimal.
\end{ex}

\begin{Remark}\label{automin}
The discussions above establish that if the space $X$ is connected, the system $(X,\mathbb{F})$ is minimal if and only if the system $(X,f)$ is minimal. Further, Example \ref{min} proves that connectedness is indeed a necessary condition and the result does not hold good in the absence of the condition stated. It may be noted that the proof of Theorem \ref{mini} does not require the maps $f_i$ in the generating family to be distinct. Thus for an autonomous system $(X,f)$, a similar proof establishes the minimality of $f^k$ (from minimality of $f$) when the space $X$ is connected. Consequently, if the space $X$ is connected, an autonomous system $(X,f)$ is minimal if and only if $(X,f^m)$ is minimal (for each $m\in\mathbb{N}$). Further, connectedness is once again a necessary condition and the derived conclusion does not hold good when the stated condition is dropped. Hence we obtain the following corollary.
\end{Remark}

\begin{Cor}
If $X$ is connected then, $(X,f)$ is minimal $\Leftrightarrow$ $(X,f^m)$ is minimal for each $m\in\mathbb{N}$. Further, there exists a disconnected space $X$ and a continuous self map $f$ on $X$ such that $(X,f)$ is minimal but $(X,f^2)$ is not minimal.
\end{Cor}

\begin{proof}
The proof follows from the discussions in Remark \ref{automin}. The conclusion follows directly from Example \ref{mini} as $(X,f_1)$ is minimal but $(X,f_1^2)$ is not minimal.
\end{proof}

\begin{Proposition}
$(x,y)$ is proximal in $(X,\mathbb{F})$ if and only if $(x,y)$ is proximal in $(X,f)$.
\end{Proposition}
\begin{proof}
Let $(x_1,x_2)$ be proximal for $(X,\mathbb{F})$ and let $(r_n)$ be the sequence of natural numbers such that $\lim \limits_{n\rightarrow\infty} d(\omega_{r_n}(x_1),\omega_{r_n}(x_2))=0$. As $X$ is compact, without loss of generality (by passing on subsequence which we again denote by $(r_n)$), we obtain an element $z\in X$ such that $(\omega_{r_n}(x_1),\omega_{r_n}(x_2))$ converges to $(z,z)$. As the family $\mathbb{F}$ is finite, there exists a subsequence $(r_{n_l})$ of $(r_n)$, $s\in\{1,2,\ldots,k\}$ and a sequence $(km_{n_l})$ (of multiples of $k$) such that $\omega_{r_{n_l}}(x_i)= f_s\circ f_{s-1}\circ\ldots\circ f_1 (\omega_{km_{n_l}}(x_i))$ (for $i=1,2$). As $\omega_{r_{n_l}}(x_i)$ converges (to $z$) and $f_k\circ f_{k-1}\circ\ldots\circ f_{s+1}$ is continuous, $f_k\circ f_{k-1}\circ\ldots\circ f_{s+1}(\omega_{r_{n_l}}(x_i))$ converges to $f_k\circ f_{k-1}\circ\ldots\circ f_{s+1}(z)$ (for $i=1,2$) or $\omega_{k(m_{n_l}+1)}(x_i)$ converges to $f_k\circ f_{k-1}\circ\ldots\circ f_{s+1}(z)$ (for $i=1,2$). Consequently, $\lim \limits_{l\rightarrow \infty} d(\omega_{k(m_{n_l}+1)}(x_1), \omega_{k(m_{n_l}+1)}(x_2))=0$ and hence $(x_1,x_2)$ is proximal for $(X,f)$.

Conversely, as orbit of any point $x$ under $(X,f)$ is a subset of orbit of $x$ under $(X,\mathbb{F})$, proximality of the pair $(x_1,x_2)$ for $(X,f)$ ensures proximality of $(x_1,x_2)$ for $(X,\mathbb{F})$ and hence the proof of converse is complete.
\end{proof}

\begin{Remark}\label{prox}
The above result establishes the equivalence of proximal pairs for the two systems $(X,\mathbb{F})$ and $(X,f)$. While proof of the converse is straightforward, the forward part uses the fact that if $k$ is a fixed natural number and $(r_n)$ is a sequence of natural numbers then there exists a subsequence $(r_{n_l})$ of $(r_n)$ such that $(r_{n_l}~~ modulo ~~k)$ is constant and hence the set of proximal pairs for the two systems coincide. Further, as equivalence of proximal pairs for two systems ensures equivalence of distal pairs, the system $(X,\mathbb{F})$ is distal if and only if $(X,f)$ is distal. Hence we get the following corollary.
\end{Remark}

\begin{Cor}
$(X,\mathbb{F})$ is distal $\Leftrightarrow ~~(X,f)$ is distal.
\end{Cor}

\begin{proof}
The proof follows from discussions in Remark \ref{prox}.
\end{proof}


\begin{Proposition}\label{sen}
$(X,\mathbb{F})$ is sensitive if and only if $(X,f)$ is sensitive.
\end{Proposition}
\begin{proof}
Let $(X,\mathbb{F})$ be sensitive (with sensitivity constant $\delta$). As the family $\{g_r=f_r\circ f_{r-1}\circ\ldots\circ f_1 : r=1,2,\ldots,k-1\}$ is finite family of continuous (uniformly continuous) maps, there exists $\eta>0$ such that $d(x,y)<\eta$ ensures $d(g_r(x),g_r(y))<\delta~~\forall x,y\in X$. We claim that $\eta$ is sensitivity constant for $(X,f)$. Note that if there exists open set $U$ such that $diam(\omega_{nk}(U))<\eta~~ \forall n$ then, $diam(g_r(\omega_{nk}(U)))<\delta$ for all $r=1,2,\ldots,k$ and $n\in \mathbb{N}$. As $\{g_r(\omega_{nk}(U)): r=1,2,\ldots,k, n\in\mathbb{N}\}$ coincides with the trajectory of the open set $U$ under $(X,\mathbb{F})$, $diam(\omega_n(U))<\delta~~\forall n\in\mathbb{N}$ which contradicts sensitivity of the system $(X,\mathbb{F})$ and hence any open set $U$ expands (to size more than $\eta$) for $(X,f)$. Thus $(X,f)$ is sensitive (with sensitivity constant $\eta$) and the proof of forward part is complete.

Conversely, as orbit of any point $x$ under $(X,f)$ is contained in orbit of $x$ in $(X,\mathbb{F})$, sensitivity of $(X,f)$ implies sensitivity of $(X,\mathbb{F})$ and hence the proof of converse is complete.
\end{proof}

\begin{Remark}\label{li}
The above proof establishes the equivalence of sensitivity for the two systems $(X,\mathbb{F})$ and $(X,f)$. While proof in one of the directions is trivial, the other direction uses the fact that any continuous function on a compact metric space is uniformly continuous. However, the proof does not preserve the sensitivity constant and hence the two systems may be sensitive with different constant of sensitivity. It may be noted that cofinite sensitivity of $(X,\mathbb{F})$ ensures cofinite sensitivity of $(X,f)$. Also, a proof similar to proof of Theorem \ref{sen} (considering the family $\{h_r=f_k\circ f_{k-1}\circ\ldots\circ f_{k-r}: r=0,1,\ldots,k-1\}$ and proving that common constant of uniform continuity is sensitivity constant for $(X,\mathbb{F})$) establishes that cofinite sensitivity of $(X,f)$ ensures cofinite sensitivity of $(X,\mathbb{F})$ and hence cofinite sensitivity is equivalent for the two systems. 
Thus we get the following corollary.
\end{Remark}

\begin{Cor}
$(X,\mathbb{F})$ is cofinite sensitive if and only if $(X,f)$ is cofinite sensitive.
\end{Cor}

\begin{proof}
The proof follows from discussions in Remark \ref{li}.
\end{proof}

\begin{Proposition}\label{lis}
$(X,\mathbb{F})$ is Li-Yorke sensitive if and only if $(X,f)$ is Li-Yorke sensitive.
\end{Proposition}
\begin{proof}

Let $(X,\mathbb{F})$ be Li-Yorke sensitive (with sensitivity constant $\delta$) and let $x\in X$. For any neighborhood $U$ of $x$, there exists $y\in U$ such that $\liminf \limits_{n\rightarrow \infty} d(\omega_n(x),\omega_n(y))=0$ and $\limsup \limits_{n\rightarrow \infty} d(\omega_n(x),\omega_n(y))>\delta$. As $\delta$ is constant of sensitivity, a proof similar to Proposition \ref{sen} ensures existence of $\eta>0$  such that $\limsup \limits_{n\rightarrow \infty} d(\omega_{nk}(x),\omega_{nk}(y))>\eta$. Further, as a pair is  proximal for $(X,\mathbb{F})$ if and only if it is proximal for $(X,f)$, $(x,y)$ is $\delta$ Li-Yorke pair for $(X,\mathbb{F})$ ensures that $(x,y)$ is $\eta$ Li-Yorke pair for $(X,f)$. Hence, $(X,\mathbb{F})$ is Li-Yorke sensitive implies $(X,f)$ is Li-Yorke sensitive and the proof of forward part is complete.

Conversely, as orbit of any point $x$ under $(X,f)$ is contained in orbit of $x$ in $(X,\mathbb{F})$, Li-Yorke sensitivity of $(X,f)$ implies Li-Yorke sensitivity of $(X,\mathbb{F})$ and hence the proof of converse is complete.
\end{proof}

\begin{Cor}
$(X,\mathbb{F})$ is Li-Yorke chaotic if and only if $(X,f)$ is Li-Yorke chaotic.
\end{Cor}
\begin{proof}
As Proposition \ref{lis} establishes preservability of Li-Yorke pairs between the systems $(X,\mathbb{F})$ and $(X,f)$, the corollary is the direct consequence of Proposition \ref{lis}.
\end{proof}

\bibliography{xbib}

\end{document}